\documentclass[a4paper,11 pt]{amsart}
\usepackage[usenames,dvipsnames,svgnames,table]{xcolor}
\usepackage[all]{xy}
 \usepackage{enumitem}
 \usepackage[normalem]{ulem}

\usepackage{tikz}
\usepackage{tikz-cd}

\makeatletter
\def\@endtheorem{\endtrivlist}
\makeatother

\usepackage[active]{srcltx}
\usepackage{amsmath}
\usepackage{mathtools}
\usepackage{amssymb}
\usepackage{amscd}
\usepackage{amsthm}
\usepackage[latin1]{inputenc}
\usepackage{mathrsfs}  

\usepackage{nicefrac}

\newtheorem{teo}{Theorem}[section]
\newtheorem{defin}[teo]{Definition}
\newtheorem{prop}[teo]{Proposition}
\newtheorem{cor}[teo]{Corollary}
\newtheorem{lemma}[teo]{Lemma}
\theoremstyle{definition}

\newtheorem{remark}[teo]{Remark}
\newtheorem{ex}{Example}

\newtheoremstyle{dico}
 {\baselineskip}   
  {\topsep}   
  {}  
  {0pt}       
  {} 
  {.}         
  {5pt plus 1pt minus 1pt} 
  {}          
\theoremstyle{dico}
\newtheorem{say}[teo]{}
\numberwithin{equation}{section}

\newcommand{\go}{\operatorname{GO}}
\newcommand{\gsp}{\operatorname{GSp}}

\newcommand{\ra}{\rightarrow}
\newcommand{\C}{\mathbb{C}}
\newcommand{\R}{\mathbb{R}}
\newcommand{\Zeta}{{\mathbb{Z}}}

\newcommand{\QQ}{{\mathbb{Q}}}
\newcommand{\meno}{^{-1}}

\newcommand{\alfa}{\alpha}

\newcommand{\vacuo}{\emptyset}

\newcommand{\La}{\Lambda}
\newcommand{\la}{\lambda}
\newcommand{\restr}[1]          {\vert_{#1}}

\newcommand{\End}{\operatorname{End}}

\newcommand{\Lie}{\operatorname{Lie}}

\renewcommand{\setminus}{-}

\newcommand{\om}{\omega}
\newcommand{\eps}{\varepsilon}
\renewcommand{\phi}{\varphi}
\newcommand{\lds}{\ldots}
\newcommand{\cds}{\cdots}
\newcommand{\cd}{\cdot}
\newcommand{\im}{\operatorname{im}}
\newcommand{\sx}{\langle}
\newcommand{\xs}{\rangle}

\newcommand{\lra}{\longrightarrow}

\newcommand{\demi}{\frac{1}{2}}

\newcommand{\ga}{\gamma}
\newcommand{\Ga}{\Gamma}

\newcommand{\rang}{\operatorname{rank}}
\newcommand{\rank}{\operatorname{rank}}

\newcommand{\lieg}{\mathfrak{g}}
\newcommand{\liem}{\mathfrak{m}}

\newcommand{\OO}{\mathcal{O}}

\newcommand{\est} {\Lambda}

\newcommand{\Gl}{\operatorname{GL}}
\newcommand{\GL}{\operatorname{GL}}
\newcommand{\gl}{\mathfrak{gl}}

\newcommand{\chern}{\operatorname{c}}   

\newcommand{\Ad}{\operatorname{Ad}}

\newcommand{\PP}{\mathbb{P}}

\renewcommand{\phi}             {\varphi}

\newcommand{\HH}{\mathfrak{H}}
\newcommand{\sieg}{\HH}

\newcommand{\Hg}{\HH_g}
\newcommand{\mt}{\operatorname{MT}}

\newcommand{\Sp}                {\operatorname {Sp}}
 \newcommand{\U}                 {\operatorname {U}}    
\newcommand{\liu}                 {\mathfrak {u}}

\newcommand{\liek}{\mathfrak{k}}

\newcommand{\liel}{\mathfrak{l}}

\newcommand{\lieq}{\mathfrak{q}}

 \newcommand{\barj}              {j}       
 \newcommand{\barf}              {\bar{f}}       
 \newcommand{\barB}              {\bar{B}}       
\newcommand{\barb}              {\bar{B}}       
 \newcommand{\barx}              {\bar{X}}       
 \newcommand{\barX}              {\bar{X}}       

\newcommand{\mg}{\mathsf{M}_g}
\newcommand{\Mg}{\mathsf{M}_g}

\newcommand{\A}{\mathsf{A}}
\newcommand{\Ag}{\mathsf{A}_g}

\newcommand{\des}{\mathbb{S}}
\newcommand{\deli}{\mathbb{S}}
\newcommand{\inv}{\mathrm{inv}}
\newcommand{\prim}{\phantom{}_\mathrm{prim}}
\newcommand{\prima}{H}

\newcommand{\Hi}{H^{\mathrm{inv}}}
\newcommand{\bul}{\bullet}
\newcommand{\hB}{\hat{B}}

\newcommand{\M}{\mathsf{M}}
\newcommand{\Mm}{\mathsf{M}^{(m)}}
\newcommand{\T}{\mathsf{T}}
\newcommand{\mb}{\overline{\M}}
\newcommand{\Mb}{\overline{\M}}
\newcommand{\bary}{\bar{Y}}
\newcommand{\Pic}{\operatorname{Pic}}
\newcommand{\mihi}[1] {} 
\newcommand{\dodo}{\Delta_1^0}
\newcommand{\Sym}{\mathscr{Z}}
\newcommand{\sime}{\Sym_g}
\newcommand{\ad}{{\operatorname{ad}}}

\newcommand*{\tras}[2][-3mu]{\ensuremath{\mskip1mu\prescript{\smash{\mathrm t\mkern#1}}{}{\mathstrut#2}}}%

\newcommand{\vr}{V\restr{B}}
\newcommand{\bi}{p}

\newcommand{\uu}{U}

\newcommand{\bombolo}{\mathbb{F}}

\begin{document}


\author{Paola Frediani, Alessandro Ghigi and Gian Pietro Pirola}

\title{Fujita decomposition and Hodge loci}

\address{Universit\`{a} di Pavia} \email{paola.frediani@unipv.it}
\email{alessandro.ghigi@unipv.it} 
 \email{gianpietro.pirola@unipv.it}

\thanks{The authors were partially supported by MIUR PRIN 2015
  ``Moduli spaces and Lie theory'' and by GNSAGA of INdAM.  The first
  author was also partially supported by FIRB 2012 `` Moduli Spaces
  and their Applications''.  }  \subjclass[2000]{
  14C30;
  14D07;
  14H10;14H15;
  14H40;
  32G20
}

\begin{abstract}
  This paper contains two results on Hodge loci in $\Mg$.  The first
  concerns fibrations over curves with a non-trivial flat part in the
  Fujita decomposition.  If local Torelli theorem holds for the fibres
  and the fibration is non-trivial, an appropriate exterior power of
  the cohomology of the fiber admits a Hodge substructure. In the case
  of curves it follows that the moduli image of the fiber is
  contained in a proper Hodge locus.  The second result deals with
  divisors in $\Mg$. It is proved that the image under the period map
  of a divisor in $\Mg$ is not contained in a proper totally geodesic
  subvariety of $\Ag$. It follows that a Hodge locus in $\Mg$ has
  codimension at least 2.


\end{abstract}

\maketitle

\tableofcontents{}

\section{Introduction}

This paper contains two results concerning Hodge loci.

\medskip

The first one relates Hodge loci to the second Fujita decomposition.
Let $\barx$ be a complex projective manifold of dimension $n+1$ and
let $\barf: \barx \lra \barb $ be a fibration onto a smooth projective
curve $\barB$.  Denote by $B$ the set of regular values of $\barf$.

Fujita decomposition says roughly that the Hodge bundle splits as a
direct sum of an ample vector bundle and a unitary flat bundle, see
\cite{fujita-fiber,kollar,cadett-1,cadett-2,cadett-3} and Section
\ref{sec:fujita}.  Let $d$ be the rank of the flat summand in the
Fujita decomposition.


Our first result is as follows.

{
\begin{teo} [See Theorem \ref{cor2}]
  \label{main1}
  Assume that $d >0$ and that for generic $b \in B$ the 
  IVHS map $T_{b}B \otimes H^{n,0}(X_b) \ra H^{n-1,1} (X_b)$ is
  non-zero.  Then for any $b\in B$ the Hodge structure
  $\est^d H^n(X_b)\prim $ admits a proper substructure.
\end{teo}
}

We notice that for $n$ odd the Hodge structure $\est^d H^n (F)\prim$
always splits as a non-trivial direct sum of Hodge substructures.  For
odd $n$ the substructure provided by our result lies in one particular
piece of $\est^dH^n(F)$, which we denote by $E_d$, see \ref{sayspsp}.

When the fibers are curves we deduce the following.

\begin{teo}[See Theorem \ref{teolocus}]
  \label{main2}
  If $n=1$ and $\barf : \barX \ra \barb$ is a non-isotrivial fibration
  with $d > 0 $, then the image of $B$ in $\Mg$ is contained in a
  proper Hodge locus of $\Mg$.
\end{teo}

The Hodge locus containing the moduli image of $B$ is defined by a
substructure of $ \est^d H^1$.  It would be interesting to investigate
the structure of such loci.  In the case $d=1$ these loci have been
studied for example in \cite
{pirola-base-number,cp,ciliberto-geer-teixidor,ciliberto-geer}.

A similar result holds for complete intersections with ample canonical
bundle, see Theorem \ref {appli}.

These results rely on some important theorems on variations of Hodge
structure due to Deligne and Schmid. These are recalled together with
some preliminary facts in Section \ref {notazione}.  The proofs of the
above results are contained in Section \ref{sec:fujita}.

In Section \ref{sec:examples} we describe the Hodge substructure
provided by the theorem in two examples due to Catanese and Dettweiler
\cite{cadett-3}.

\medskip

The second part of the paper deals with Hodge loci from another
perspective.  Let $j : \Mg \ra \Ag$ be the period map. A Hodge locus
of $\Mg$ is nothing else than $j\meno(Z)$ for a Hodge locus
$Z \subset \Ag$.  Hodge loci in $\Ag$ have an important property: they
are totally geodesic subvarieties of $\Ag$, when $\Ag$ is endowed with
the Siegel metric, i.e. it is considered as a locally symmetric
orbifold.

The main result of Section \ref{sec:hodge-moduli} is the following.
\begin{teo}[See Theorem \ref{notg}]
  \label{main3}
  If $g\geq 3$ and $Y \subset \M_g$ is an irreducible divisor, then
  there is no proper totally geodesic subvariety of $\Ag$ containing
  $j(Y)$.
\end{teo}
In particular we have the following corollary:
\begin{cor} [See Corollary \ref{hlmg}]
  If $g \geq 3$, any Hodge locus of $\Mg$ has codimension at least 2.
\end{cor}

The proof of Theorem \ref{main3} is based on a result of independent
interest (Theorem \ref{pietro}) that describes the behaviour of a
divisor in $\Mg$ at the boundary. This result is a variation on an
argument in \cite{marcucci-naranjo-pirola}.  It allows to use
induction on $g$.  The case $g=3$ follows, as a very special case,
from a theorem by Berndt and Olmos \cite{berndt-olmos} on the
codimension of totally geodesic submanifolds in symmetric spaces, see
\ref{bo}.  The inductive step depends on simple Lie theoretic
computations showing that $\sieg_k \times \sieg_{g-k} $ is a maximal
totally geodesic submanifold of $\sieg_g$, see Proposition \ref{oops}.

\medskip

{\bfseries \noindent{Acknowledgements}}.  The authors wish to thank
Professor Atsushi Ikeda for introducing them to paper \cite{schoen}.
The second author would like to thank Professor J.S. Milne for
interesting information on arithmetic varieties.

\section{Preliminaries and notation}
\label{notazione}

We start by recalling the main definitions related to Hodge theory
needed in the paper.
Let $H$ be a rational vector space of finite dimension.  Set
$ H_\R := H\otimes_\QQ \R$ and $ H_\C := H\otimes_\QQ \C$.
\begin{defin}
  A rational Hodge structure of weight $n$ is the datum of (1) a
  $\QQ$-vector space $H$ and (2) a decomposition
  $ H_\C := \bigoplus_{p+q=n} H^{p,q}, $ such that
  $H^{q,p} = \overline{H^{p,q}}$.
\end{defin}
Set
$ \des := \{ A=(a_{ij}) \in \Gl(2): a_{11} = a_{22} , a_{12} + a_{21}
= 0 \}$.
$\des$ is an algebraic group defined over $\QQ$.  The map
$ A \mapsto z:= a_{11}+ i a_{21}, $ is a group isomorphism
$\des(\R) \cong \C^*$.  The maps $ \phi_\pm : \des (\C) \lra \C^*$,
defined by $ \phi_\pm (A) : = a_{11} \pm i a_{21} $, are characters of
$\des(\C)$ and $f = (\phi_+, \phi_-)$ is an isomorphism of $\des(\C)$
onto $\C^*\times \C^*$.
It follows that every character of $\des(\C)$ is of the form
$\phi_+^p \cd \phi_-^q$.  If $\mathbb{G}_m$ denotes the multiplicative
group scheme, then $ w : \mathbb{G}_m \ra \des$, $ w(a) : = a I_2$
is an injective morphism defined over $\QQ$.
\begin{say}
  A $\QQ$-Hodge structure of weight $n$ is equivalent to the datum of
  a $\QQ$-vector space $H$ with a representation
  $\rho : \des(\R) \lra \Gl(H_\R)$ such that
  $\rho \circ w (a) = a^{n}$.  Indeed
  $\rho^\C : \des(\C) \lra \Gl(H_\C)$
  splits as a sum of eigenspaces $H^{p,q}$, on which the action is
  multiplication by the character $\phi_+^p \cd \phi_- ^q$.
  Since $\rho^\C(z, z) = \rho^\C \circ w^\C (z)$, it follows that
  $H^{p,q} \neq \{0\}$ only if $p+q = n$.
  It is easy to check that $\overline{H^{p,q}} = H^{q,p}$.  Thus we
  get a $\QQ$-Hodge structure.
\end{say}

%
\begin{defin}
  A \emph{polarization} on a Hodge structure $(H, H^{p,q})$ of weight
  $n$ is a bilinear form $Q: H \times H \ra \QQ$ with the following
  properties.
  \begin{enumerate}
  \item $Q(a, b) = (-1)^n Q(b,a)$.
  \item If $h: H_\C \times H_\C \ra \C$ is the Hermitian form
    $ h(a,b):= i^n Q(a, \bar{b})$,
    then $h(H^{p,q}, H^{p',q'})=0$ if $(p,q) \neq (p',q')$.
  \item The restriction of $h$ to $H^{p,q}$ is definite of sign
    $(-1)^q (-1)^{n (n-1)/2}$.
  \end{enumerate}
 
 \end{defin}
 If $H^*\otimes H^*$ is endowed with the induced Hodge structure,
 $Q\in H^*\otimes H^*$ is an element of type $(-n,-n)$.

 \begin{say}
   \label{saymt}
   Let $\rho : \des(\R) \ra \Gl(H_\R)$ be a Hodge structure of weight
   $n$.  The \emph{Mumford-Tate group} of $H$, denoted $\mt(H)$, is
   the smallest $\QQ$-algebraic subgroup of $\Gl(H)$ whose real points
   contain $\im \rho$.  The main property of the Mumford-Tate group is
   the following: given multi-indices $d, e \in \mathbb{N}^m$ consider
   \begin{gather*}
     T^{d,e} (H):= \oplus_{j=1}^m H^{\otimes d_j} \otimes
     (H^*)^{\otimes e_j}.
   \end{gather*}
   This space is a sum of pure Hodge structures.  A (rational) vector
   $v \in T^{d,e}(H)$ is invariant by the natural action of $\mt(H) $
   if and only if it is a Hodge class of type $(0,0) $.  See
   \cite{ggk,moonen-oort,schnell-trento,voisin-hodge-loci} for more
   details.
 \end{say}

 \begin{say}
   If $n$ is odd, a polarization $Q$ is a symplectic form on $H$, if
   $n$ is even it is a non-degenerate symmetric form. Let $\gsp(H) $ be
   the group of symplectic similitudes in case $n$ is odd and let
   $\go(H)$ be the group of orthogonal similitudes for $n$ even. They
   are algebraic subgroups defined over $\QQ$.  If $Q $ is a
   polarization for the Hodge structure $H$, then
   $\mt(H) \subset \gsp(H)$ for $n$ odd and $\mt(H) \subset \go(H)$
   for $n$ even.
 \end{say}

\begin{say}\label{ortosay}
  If $(H, H^{p,q})$ is a rational Hodge structure, a \emph{(rational)
    Hodge substructure} is a rational subspace $L \subset H$, such
  that
  \begin{gather*}
    L = \bigoplus L^{p,q}, \quad \text{with } L^{p,q} := L \cap
    H^{p,q}.
  \end{gather*}
  Since $L = \bar{L}$ it follows that $(L, L^{p,q})$ is a Hodge
  structure.  If $ Q$ is a polarisation on $(H,H^{p,q})$ and
  $L \subset H$ is a Hodge substructure, then
  \begin{gather*}
    L^\perp = \{a \in H: Q(a, b) = 0, \text{ for all } b\in L\},
  \end{gather*}
  is a Hodge substructure as well and $H = L \oplus L^\perp$.
We say that a Hodge substructure $L \subset H$ is \emph{proper} if $\{0\} \neq L \subsetneq H$.
\end{say}

\begin{say} \label{projector} Let $(H,H^{p,q}, Q)$ be a polarised
  Hodge structure. Let $p \in (\End H)^{0,0}$ be a projector
  (i.e. $p^2 = p$) defined over $\QQ$. Then $\ker p$ and $\im p$ are
  substructures and $H= \ker p \oplus \im p$ as a Hodge structure.
  Conversely, if $H ' \subset H$ is a Hodge substructure, the
  orthogonal projection $p: H_ \C \ra H_ \C$ onto $H'_\C$ is a morphism
  of Hodge structures, i.e. $p \in (\End H)^{0,0}$.
\end{say}

\begin{defin}
  If $H$ is a local system of $\QQ$-vector spaces on a complex
  manifold $B$, a \emph{subsystem} of $H$ is a subsheaf of
  $\QQ$-subspaces of $H$.
\end{defin}

\begin{defin}
  If $(B, H, F^\bul)$ is a variation of Hodge structure and
  $H' \subset H$ is a subsystem, we say that $H'$ is a
  \emph{subvariation} if for any $b\in B$ the module $H_b'$ is a Hodge
  substructure of $(H_b, F^\bul_b)$.
\end{defin}

If $(B,H, F^\bullet )$ is a variation of Hodge structure and $b\in B$,
we denote by $\Hi_b$ the subspace of $H_b$ containing the vectors
that are invariant by the monodromy action $\pi_1(B,b) \ra \Gl(H_b)$.
As $b$ varies in $B$, the modules $\Hi_b$ describe a subsystem $\Hi$
of $H$.  By construction every section of $H$ takes values in $\Hi$
and the evaluation map $\Gamma(H) \mapsto H_b$ is an isomorphism onto
$\Hi_b$.

We call the subsystem $\Hi \subset H$ the \emph{fixed part} of $H$.

\begin{teo}
  [Schmid]
  \label{schmid}
  Let $(B,H,F^\bul , Q) $ be a polarized variation of Hodge structure
  on a quasi-projective manifold $B$.  (1) Let $s$ be a global flat
  section of $H_\C$. Then the $(p,q)$-component of $s$ is also flat.
  (2) The fixed part $H^\inv$ is a subvariation of $H$.
\end{teo}
The first statement is Thm. 7.22 in \cite {schmid-sing}. The second
statement follows since complex conjugation sends flat sections to
flat sections.

\begin{teo}
  [Deligne]
  \label{deligne}
  If $(B,H, F^\bul,Q)$ is a polarized variation of Hodge structure on
  a quasi-projective manifold $B$ and $G \subset H_\C$ is a subsystem of
  dimension $d$, then there is $m>0$ such that
  $(\est ^d G)^{\otimes m} = \C_B$.
\end {teo}
This is proved in \cite{deligne-2}, Cor. 4.2.8 (iii), b). The proof
uses the fact that the category of variations of Hodge structures on a
quasi-projective manifold satisfies the properties (4.2.2.1) -
(4.2.2.4). The only non-trivial point is property (4.2.2.4) which is
exactly Schmid Theorem \ref{schmid}.  See \cite[p. 45,
note]{deligne-2}.

\begin{say}
  We say that a variation of Hodge structure $(B,H,F^\bullet)$
  is \emph{trivial} if both $H$
  and $F^\bullet$
  are product bundles. We say that it is \emph{isotrivial} if there is
  a finite branched cover $f:B'\ra
  B$ such that $(B', f^*H, f^*F^\bullet)$ is trivial.
\end{say}

\begin{say}\label{saylocus}
  Given a variation of Hodge structure $(B,H,F^\bullet)$ and
  multi-indices $d, e\in \mathbb{N}^m$, passing to the universal cover
  $p: \tilde{B} \ra B$ we have
  $p^*(T^{d,e}H) = \tilde{B} \times T^{d,e} H_0$ for a fixed
  $\QQ$-vector space $H_0$. A vector $t \in T^{d,e}H_0$ gives a
  section $\tilde{t}$ of $p^*(T^{d,e}H)$.  Consider the locus
  $Y(t):=\{\tilde{b}\in \tilde{B}: \tilde{t}(\tilde{b}) \in
  (p^*T^{d,e}H)^{0,0} \}$.
  It is a countable union of irreducible analytic sets in $\tilde{B}$.
  A \emph{Hodge locus} is by definition an irreducible analytic subset
  $Z \subset B$ such that there exist 
  $t_i \in T^{d_i, e_i} H_0$ for $i=1, \lds, k$, such that $Z$ is an
  irreducible component of
  $p ( Y(t_1) \cap \cds \cap Y(t_k) ) \subset B$.  See \cite
  {moonen-oort,voisin-hodge-loci} for more details.

\end{say}

\begin{defin}
  The Hodge loci for the natural variation of Hodge structure on $\Ag$
  are called \emph{special subvarieties} or \emph{Shimura subvarieties}.
\end{defin}

\begin{say}
  \label{sayspsp}
  We recall some elementary facts from representation theory that are
  needed in Section \ref{sec:fujita}.  Let $ W$ be a complex vector space of
  dimension $2m$ with a complex symplectic form $\om$.  The symplectic
  form induces an isomorphism $W \cong W^*$.  Let $\omega^\flat $ be
  the element of $ \est^2 W$ corresponding to $\om \in \est^2 W^*$.
  For $2\leq p \leq m$ set
  \begin{gather*}
    \phi_p : \est^p W \ra \est^{p-2} W, \quad \phi_p (s): = \om
    \lrcorner s,
  \end{gather*}
  and $\phi_1 \equiv 0$.  For $p\geq 2$ the morphism $\phi_p$ is
  surjective and $E_p(W) : = \ker \phi_p$ is an irreducible
  representation of $\Sp(W)$.  Setting $L: \est^iW \ra \est^{i+2}W$,
  $L(s) : = \om^\flat \wedge s$, we have
  \begin{gather}
    \label{pseudolepeschetz}
    \est^p W = E_p(W) \oplus L(\est^{p-2}W)
  \end{gather}
  as $\Sp(W)$-modules.  It is clear that the group $\gsp(W)$ also
  preserves $E_p(W)$, hence we get an irreducible representation
  \begin{gather*}
    \eps_p : \gsp(W) \ra \Gl(E_p(W)).
  \end{gather*}
  See e.g. \cite[p. 14]{sato-kimura}, \cite[p. 260]{fulton-harris}
  \cite [p. 201]{bourbakotto} for the proof and more details.
\end{say}

\begin{say}
\label{grunfigrunfi}
  Let now $(H,Q)$ be a polarized rational Hodge structure of weight
  $n$. Assume that $n$ is odd and positive. For $p \leq \dim H/2$,
  $\est^p H$ is a Hodge structure and $E_p(H)$ is a substructure.  The
  same holds for $L(\est^{p-2}H)$.  Indeed if
  $\rho : \deli \ra \gsp(H_\C) $ is the representation defining the
  Hodge structure on $H$, then $\eps_p \circ \rho$, which is a summand
  of $\est^p \rho$, defines $E_p(H)$ as a substucture of $\est^p H$.
  We notice that if $n=1$ and $H$ corresponds to an abelian variety
  $A$, then $E_p(H) = H^p(A)\prim$.
\end{say}

\begin{say}\label{sayzero}
  In the setting of \ref{grunfigrunfi} we have
  \begin{gather*}
    (\est^ p H)^{np,0} \subset E_p(H).
  \end{gather*}
  Indeed given $\alfa \in (\est^ p H)^{np,0}$, write
  $\alfa = \beta + \gamma$, with $\beta \in E_p(H)$ and
  $\ga \in L(\est^{p-2}H)$.  Taking the $(np,0)$-components we get
  $\alfa = \beta^{np,0} + \gamma^{np,0}$ with
  $ \beta^{np,0} \in E_p(H)$ and  $ \gamma^{np,0} \in L(\est^{p-2} H)$,
  since both $E_p(H)$ and $L(\est^{p-2}H)$ are substructures.  The
  operator $L $ is the wedge with $Q^\flat$.  Since
  $Q \in (H^*\otimes H^*)^{-n,-n}$,
  $Q^\flat \in (H \otimes H)^{n,n}$.  Hence
  $(L(\est^{p-2}H))^{np,0} = \{0\}$.  Therefore
  $\alfa = \beta^{np,0}$.
\end{say}

\begin{say}
  \label{oo}
  Let $W$ be a complex vector space of dimension $m$ with a
  non-degenerate symmetric bilinear form $q$.  We recall that if $m$
  is odd then $\est^pW$ is an irreducible representation of $O(W,q)$
  for any $p$.  If $m=2s$, then $\est^pW$ is an irreducible
  representation of $O(W,q)$ for any $p \neq s$.  See \cite[p. 287 and
  p. 295]{fulton-harris} or \cite[p. 16 and p. 18]{sato-kimura}.  It
  follows that if $(H,Q)$ is a polarised rational Hodge structure of
  weight $n$, with $n$ even, then $\est^p H$ is an irreducible
  representation of $\go(H)$ for any $p\neq \dim H /2$.
\end{say}

\section{Fujita decomposition and Hodge loci}
\label{sec:fujita}

\begin{say}
  \label{mainsay}
  Let $\barx$ be a complex projective manifold of dimension $n+1$ and
  let $\barf: \barx \lra \barb $ be a fibration onto a smooth
  projective curve $\barB$.  Denote by $B$ the set of regular values
  of $\barf$, set $X:=\barf\meno (B)$ and $f:=\barf\restr{X}$.

  Fix a Hodge class $[\om] \in H^2(\barx, \Zeta)$.  Consider the
  fibrewise primitive cohomology with respect to the restriction of
  this polarisation to the smooth fibers:
  \begin{gather*}
    \prima := (R^n f_*\QQ_X)\prim.
  \end{gather*}
  Then $\prima$ is a local system of $\QQ$-vector spaces on $B$.  Its
  associated vector bundle $H\otimes \OO_B$ is endowed with the
  Gauss-Manin connection $\nabla$ and with a polarization $Q$ obtained
  from the intersection form.  Denote by $F^\bullet$ the weight
  filtration.  Then $(B,H, F^\bullet, Q)$ is a polarised variation of
  Hodge structure on the quasiprojective curve $B$.  Let $h$ denote
  the associated Hermitan form:
  $h (\alfa, \beta) = i^n Q(\alfa, \bar{\beta} )$.  The sheaf
  \begin{gather*}
    V:= \barf_* \om_{\barx /\barb}
  \end{gather*}
  is locally free on $\barb$, so we identify it with the corresponding
  vector bundle $V\ra \barb$, which is called the \emph{Hodge bundle}
  of the fibration.  We have $V\restr{B} = F^n$. The restriction of
  $h$ to $\vr$ is positive definite, hence we can define the
  orthogonal projection $p : H_\C \ra \vr$. The Gauss-Manin connection
  induces a connection $D:=p\nabla $ on $\vr$.
\end{say}

\begin{teo}
  [Fujita] In the setting of \ref{mainsay} there
  is a decomposition
  \begin{gather}
    \label{Fujita}
    V=A\oplus U,
  \end{gather}
  where $A$ in an ample vector bundle on $\barB$ and $U$ is a vector
  bundle on $\barB$, such that $U\restr{B}$ is the holomorphic vector
  bundle associated to a subsystem of $H$.  We will call
  \eqref{Fujita} the \emph{(second) Fujita decomposition}.
\end{teo}
See \cite{fujita-fiber,fujita-bis} for background and
\cite{kollar,cadett-1,cadett-2} for the proof. See also
\cite{cadett-3,pirola-torelli,vls} for related problems.

\begin{say}
  \label{spezzocon}
  On $B$ we have $D( U ) \subset U \otimes \Omega^1_B$ and
  $D( A) \subset A \otimes \Omega^1_B$, so $D=D^A \oplus D^U$.  Indeed
  let $\sigma_U :U \ra A \otimes \est^{1,0} $ and
  $\sigma _A : A \ra U \otimes \est^{1,0}$ be the second fundamental
  forms of $ U \subset V$ and $A \subset V$.  By
  \cite[p. 20]{kobayashi-vector} $\sigma_A$ is (up to sign) the
  adjoint of $\sigma _U$. Moreover $\sigma_U=0$ since $U$ is preserved
  by $\nabla$, hence by $D$.
\end{say}

The following observation is well-known.
\begin{lemma}
  \label{lemma7}
  If $(V\restr{B}, D)$ is flat, then $\vr$ is preserved by $\nabla$.
\end{lemma}
\begin{proof}
  We use the notation of
  \cite[p. 224]{arbarello-cornalba-griffiths-2}.  Let $s$ be a local
  holomorphic section of $\vr = F^n \subset H_\C$.  Then
  $ i \sx R_Vs,s \xs = i \sx R_{H_\C} s, s\xs -i \sx \sigma s, \sigma
  s\xs$,
  where $\sigma$ is the second fundamental form of $\vr \subset H_\C$.
  By assumption $R_V = 0$. Also $R_{H_\C}=0$. So
  $\sx \sigma s , \sigma s \xs =0$.  By Griffiths transversality
  $\nabla (\vr) = \nabla (F^n) \subset F^{n-1} \otimes \Omega^1_B$.
  Since $F^{n-1} = F^n \oplus H^{n-1,1}$, we have
  $\sigma : \vr \ra H^{n-1,1} \otimes \est^{1,0}(B)$.  But $h$ is
  definite on $H^{n-1,1}$, so we conclude that $\sigma s
  =0$. Therefore $\sigma =0$ and $\nabla$ preserves $\vr$.
\end{proof}

Our first result is the following.

\begin{teo}
  \label{main}
  Let $\barf:\barX \ra \barB$ be as in \ref{mainsay} and let
  $d:= \rang U >0$.  Then there is an \'etale cyclic cover
  $u:\hB\ra B$ such that $E_d(u^*H)^\inv \neq \{0\}$ if $n$ is odd and
  $(\est^d u^*H )^\inv \neq \{0\}$ if $n$ is even.
\end{teo}

\begin{proof}
  By assumption there is a subsystem $G \subset H_\C$ such that
  $U\restr{B}$ is the vector bundle associated to $G$.  By Theorem
  \ref{deligne} there is an $m>0$ such that
  $(\est^d G ) ^{\otimes m} \cong \C_B$.  Fix $b\in B$ and let
  $\rho : \pi_1(B, b) \ra \C^* = \Gl(\est^d G_b)$ be the monodromy
  representation of $\est^d G$. The image of $\rho$ is finite cyclic.
  Let $u:\hat{B} \ra B$ be the unramified covering associated to
  $\ker \rho \subset \pi_1(B,b)$. It is a cyclic Galois cover.  Then
  $\hat{G} : = u^* G$ is a subsystem of $u^*H $.  Fix
  $\hat{b} \in u\meno (b)$.  The monodromy representation of
  $\est^d \hat{G} =u^*\est^d G$ at $\hat{b}$ is the composition
  $\rho \circ u_* : \pi_1(\hB, \hat{b}) \ra \C^*$. Hence
  $\est^d \hat{G} \cong \C_{\hB}$.  It follows that $\est^d \hat{G}$
  is contained in the fixed part of $\est^d u^* H$.  If $n$ is even we
  are done.  If $n$ is odd, observe that
  $\est^d \hat{G} \subset (\est^d u^*H)^{nd,0} \subset E_d(u^* H)$ by
  \ref{sayzero}.  Hence $ (E_d(u^* H) )^\inv \neq \{0\}$.
\end{proof}

\begin{remark}
  It is important to stress that the variation $u^*\est^d H$ is
  \emph{geometric} and that this yields another proof of the Theorem,
  which avoids Schmid theorem.  To see this, we work over $\barB$
  instead of $B$. Up to base change we can assume that $\det U$ is
  trivial.  Set $Z:= X \times_{\barb} \cdots \times_{\barb}{X} $ ($d$
  times) and let $F: Z \ra \barB$ be the induced fibration:
  $F(x_1, \lds, x_d ) = f(x_i)$.  Over $B$ the morphism $F$ is smooth.
  Denoting by $p_i$ the $i$-th projection, the form
  $\tilde{\om} : = \sum_{i=1}^ d p_i^* \om$ is a K\"ahler form on
  $X \times \cds \times X$.  Let $R^{nd}F_*\QQ_Z\prim$ denote the
  fibrewise primitive cohomology with respect to $[\tilde{\om}]$,
  which is a geometric variation of Hodge structure on $B$.  For any
  $t\in B$, we have
  $ \est^d H ^n (X_t)\prim \subset H^{nd} (Z_t)\prim$.  Thus
  $\est^d H = \est ^d R^1f_*\QQ_X\prim$ is a summand of the geometric
  variation $R^{nd}F_*\QQ_Z\prim$ as claimed.  Next set
  \begin{gather*}
    W_t : = H^{nd}(Z_t)^\inv \cap \est^d H^n (X_t)\prim,
  \end{gather*}
  where $H^{nd}(Z_t, \QQ)^\inv $ denotes the fixed part of the
  variation $R^{nd}F_*\QQ_Z\prim$.  We can deduce that this fixed part
  is a Hodge substructure, from the Global Invariant Cycle Theorem,
  which asserts that it coincides with the image of the restriction
  $H^{nd}(Z) \ra H^{nd}(Z_t)$, see e.g. \cite[Thm. 16.24,
  p. 385]{voisin-libro}.  Thus $W_t$ is a Hodge substructure and
  $\det U \subset W$, which implies that $W$ is non-trivial.

\end{remark}

\begin{say}
  To make some of the following statements simpler we introduce the
  following notation: if $H$ is a polarized rational Hodge structure
  or a polarized rational variation of Hodge structure we set
  \begin{gather*}
    K_d(H):=
    \begin{cases}
      E_d(H) & \text{if  $n$ is odd,}\\
      \est^d H & \text{if $n$ is even.}
    \end{cases}
  \end{gather*}
\end{say}

\begin{cor}
  \label{bau}
  Let $\barf:\barX \ra \barB$ be as in \ref{mainsay} and let
  $d:= \rang U >0$.  Then either $K_d (H)$ is isotrivial or for any
  $b\in B$ the Hodge structure $K_d(H_b) $ admits a proper
  substructure.
\end{cor}
\begin{proof}
  Let $u:\hB \ra B$ be as in the Theorem \ref{main}.  Set $K:=K_d(H)$.
  Consider the variation of Hodge structure $u^*K$ on $\hat{B}$.  By
  Schmid theorem \ref{schmid} $(u^* K)^\inv$ is subvariation of $u^*K$
  and we know from Theorem \ref{main} that $(u^*K)^\inv\neq \{0\}$.
  If $(u^*K) ^\inv = u^*K$, then $u^*K$ is trivial and hence $K$ is
  isotrivial.  Otherwise for any $b\in B$ and $\hat{b} \in u\meno(b)$
  we have
  \begin{gather*}
    \{0\} \neq (u^*K)^\inv_{\hat{b}} \subsetneqq (u^*K)_{\hat{b}}=K_b.
  \end{gather*}
\end{proof}

\begin{lemma}
  Let $\barf:\barX \ra \barB$ be as in \ref{mainsay} and let
  $d:= \rang U >0$.
  \label{lemure1} If $K_d(H)$ is isotrivial, then the Hodge bundle
  $V\restr{B}$ is flat.
\end{lemma}
\begin{proof}
  Step 1: $\est^d\vr $ is preserved by $\est^d\nabla$.\\
  \noindent (Here $\est^d\nabla$ denotes the connection induced by
  $\nabla$ on $\est^d H$.)  It is easy to check that
  $(\est^d H)^{nd,0} = \est^dF^n$.  On the other hand it follows from
  \ref{sayzero} that $(\est^d H)^{nd,0} = K_d(H)^{nd,0}$.  Since
  $u^*K_d(H)$ is trivial for some base change $u: \hB \ra B$, its
  $(nd,0)$-component is preserved by $\est^d\nabla$, so also
  $K_d(H)^{nd,0} = \est^d F^n=\est^d \vr$ is preserved by
  $\est^d\nabla$.  \\ \noindent Step 2: $(\est^d \vr, \est^d D)$ is a
  flat bundle.  \\ \noindent Let $h$ denote the Hodge Hermitian
  product on $H_\C$ defined in \ref{mainsay}. Recall that $h >0$ on
  $F^n$ so the orthogonal projection $p: H_\C \ra F^n$ is
  well-defined.  Since
  $ (\est^d h) \restr{\est^d F^n} = \est^d ( h \restr{F^n}) $ and the
  right hand side is positive definite, also $\est^d h >0$ on
  $\est^d F^n$.  So the orthogonal projection
  $p': \est^d H_\C \ra \est ^d F^n$ is well-defined.  Moreover
  $ p' (\est^d \nabla) = \est ^ d D $, since both are compatible
  connections on the Hermitian bundle
  $(\est^d F^n, \est^ d h \restr{\est^d F^n})$.  Since
  $\est ^d F^n=\est^d \vr$ is preserved by $\est^d \nabla$, we
  conclude that $\est^d D = \est^d \nabla $ on $\est^d\vr$ and of
  course $\est^d D$ is flat on $\est^d \vr$.  This proves the claim.
  \\ \noindent Step 3 :
  $(\est^{d-1} U\restr{B} \otimes A \restr{B}, \est^d D)$ is flat.
  \\ \noindent $\est ^d \vr$ is a direct sum of various bundles, one
  of them being $\est^{d-1} U\restr{B} \otimes A \restr{B}$.  All
  these summands are preserved by the connection $\est^d D$, as
  follows from \ref{spezzocon}. Since $(\est^d \vr, \est^d D)$ is
  flat, every summand is flat with this connection.  \\ \noindent Step
  4: $(A\restr{B}, D^A)$ is flat.  \\ \noindent On
  $\est^{d-1} U\restr{B} \otimes A \restr{B}$ we have
  $ \est^d D = \est^{d-1} D^U \otimes D^A$.  Moreover $U$ is flat. So
  step 3 yields that
  \begin{gather*}
    0= R_{\est^{d-1}U \otimes A}= R_{\est^{d-1} U } \otimes I_A +
    I_{\est^{d-1}U } \otimes R_A = I_{\est^{d-1}U } \otimes R_A.
  \end{gather*}
  So $R_A= 0$ on $B$.  Thus $(\vr, D)$ is flat.
\end{proof}

{Recall that if $b_0\in B$ and $W \subset B$ is an open contractible
  neighbourhood of $b_0$, then the period mapping
  $\mathscr{P} ^{n,n} : W \ra \operatorname{Grass} (h^{n,0},
  H^n(X_{b_0}))$ associates to $b\in W$ the subspace
  $H^{n,0}(X_b) \subset H^n(X_b) \cong H^n(X_{b_0})$, see e.g. 
  \cite[p. 229]{voisin-libro,carlson-green-griffiths-harris}.
  The derivative of the period mapping at $b$ can be identified with the \emph{infinitesimal variation of Hodge structure}  (IVHS) map
 $T_{b}B \otimes H^{n,0}(X_b) \ra H^{n-1,1} (X_b)$ obtained by taking the cup product with the Kodaira-Spencer class.  
}

{
\begin{teo}\label{cor2}
  Let $\barf:\barX \ra \barB$ be as in \ref{mainsay} and let
  $d:= \rang U >0$.  Assume that for generic $b \in B$ the
  IVHS map
  $T_{b}B \otimes H^{n,0}(X_b) \ra H^{n-1,1} (X_b)$ is non-zero.  Then
  for any $b\in B$ the Hodge structure $K_d (H_b) $ admits a proper
  substructure.
\end{teo}
}
\begin{proof}
  By Corollary \ref{bau} either $K_d (H)$ is isotrivial or for any
  $b\in B$ the Hodge structure $K_d(H_b) $ admits a proper
  substructure.  In the first case the Hodge bundle $\vr$ is flat by
  Lemma \ref{lemure1}.  Hence by Lemma \ref{lemma7} $\vr$ is preserved
  by $\nabla$.  {But then the map
  $T_{b}B \otimes H^{n,0}(X_b) \ra H^{n-1,1} (X_b)$ would be trivial
  for any $b\in B$, contrary to the assumption.  This proves that for
  any $b\in B$ the Hodge structure $K_d(H_b) $ admits a proper
  substructure.}
\end{proof}

If $C$ is a smooth curve we denote by $\mt(C)$ the Mumford-Tate group
of $H^1(C)$.  The following fact is well-known. We recall the proof
for the reader's convenience.
\begin{prop}
  \label{vgmt}
  If $[C]$ is very general in $\Mg$, then the Mumford-Tate group
  $\mt(C) $ $ =$ $ \gsp(H^1(C,\QQ))$.
\end{prop}
\begin{proof}
  Fix $m\geq 3$. Let $\Ga_g $ be the mapping class group and let
  $\Ga_g[m]$ be the kernel of the composition
  $\Ga_g \ra \Sp(2g, \Zeta) \ra \Sp(2g, \Zeta / m)$. The moduli space
  $\M_g^{(m)}$ of genus $g$ curves with level $m$ structure is the
  quotient of the Teichm\"uller space $\T_g$ by the properly
  discontinuous and free action of $\Ga_g[m]$.  Over $\M_g^{(m)}$
  there is a universal family $\pi:\mathscr{C} \ra \Mm_g$ and a
  corresponding integral variation of Hodge structure on
  $R^1\pi_* \Zeta$.  Set $H:= H^1(C, \Zeta)$ and $ \mt (C) = \mt(H_\QQ)$,
  where $t:=[C]$ is very general in $\Mg^{(m)}$.  The image of the
  monodromy map $\rho: \pi_1(\Mm_g , t) \cong \Ga_g[m] \ra \Gl(H) $ is
  a finite index subgroup of $\Sp(H)$, see e.g.
  \cite[p. 170]{farb-margalit}.  Since $t$ is very general, it follows
  from Cor. 4.12 in \cite{voisin-hodge-loci} that $\mt(C)$ contains a
  finite index subgroup $\Ga'$ of $\im \rho$.  Then $\Ga'$ has finite
  index in $\Sp(H)$, thus it is an arithmetic subgroup of $\Sp(H)$.
  By Borel density theorem (see e.g \cite{borel-density-Crelles},
  \cite[p. 205]{platonov-rapinchuk}) $\Ga'$ is Zariski dense in
  $\Sp(H_ \QQ)$. Thus $\mt(C) $ contains $\Sp(H_ \QQ)$.
\end{proof}

\begin{prop}
  \label{312} Let $C$ be a curve with $\mt(C) =\gsp(H)$.  If
  $\barf : \barX \ra \barb$ is a fibration as in \ref{mainsay} with
  $C \cong X_t$ for some $t\in B$, then either the fibration is
  isotrivial or $U=\{0\}$ in \eqref {Fujita}.
\end{prop}
\begin{proof}
  Assume that $d:=\rank U >0$.  Here the weight $n=1$, so
  $K_d(H) = E_d(H)$. Since $\mt(C) = \gsp(H)$, $E_d(H)$ is an
  irreducible representation of $\mt(C)$, see \ref{sayspsp}.  So it is
  an irreducible Hodge structure.  {Assume by contradiction that
    $f$ is not isotrivial.  The assuption of Theorem \ref {cor2} is
    satisfied by Torelli theorem. But Theorem \ref{cor2} implies that
    the Hodge structure is reducible, which gives a
    contradiction. Therefore $f$ is isotrivial.}
\end{proof}

\begin{teo}
  If $C$ is a very general curve in $\Mg$ and
  $\barf : \barX \ra \barb$ is a fibration as in \ref{mainsay} with
  $C \cong X_t$ for some $t\in B$, then either the fibration is
  isotrivial or $U=\{0\}$ in \eqref {Fujita}.
\end{teo}
\begin{proof}
  By Proposition \ref{vgmt} the Mumford-Tate $\mt(C) $ equals
  $\gsp(H)$. The result follows from Proposition \ref{312}.
\end{proof}

The following statement refines the previous one using the notion of
Hodge locus.
\begin{teo}\label{teolocus}
  Let $\barX$ be a surface and let $\barf : \barX \ra \barb$ be a
  fibration as in \ref{mainsay}. Assume that $\barf$ is not isotrivial
  and that $d :=\rank U > 0 $. Then the image of $B$ in $\Mg$ is
  contained in a proper Hodge locus $Z$ of $\Mg$.
\end{teo}
\begin{proof}
  Fix $b \in B$ and let $L \subset E_d(H_b)$ be the Hodge substructure
  given by Corollary \ref{bau}. The orthogonal complement $L^\perp$ is
  also a Hodge substructure, so the orthogonal projection
  $p: E_d(H_b) \ra E_d(H_b)$ onto $L$ is a Hodge class in
  $(\End E_d(H_b))^{0,0}$.  Let $Z_b$ be the Hodge locus defined by
  this class, see \ref{saylocus}.  Thus the image of $B$ is contained
  in $\bigcup _b Z_b$. Since $B$ is irreducible and the Hodge loci can
  be at most countable, Baire theorem implies that the image of $B$ is
  in fact contained in $Z_{b_0}$ for some $b_0 \in B$.  It remains to
  show that $Z=Z_{b_0}$ is a proper subset of $\Mg$.  This follows
  from Proposition \ref{vgmt}: for a very general curve $C$ the space
  $E_d(H)$ is an irreducible representation of $\mt(C)$ (see
  \ref{sayspsp}), thus it is an irreducible Hodge structure.
\end{proof}

Let $\phi:\mathfrak{X} \ra S$ be the universal family of smooth
complete intersections of multidegree $d_1, \lds, d_k$ in $\PP^N$.

\begin{teo}
  \label{appli}
  Assume that $\sum_i d_i > N+1$ (i.e. the canonical bundle is
  ample). Let $s$ be a very general point of $S$.  If
  $\barf : \barX \ra \barb$ is a fibration as in \ref{mainsay} with
  $ X_t \cong \phi\meno(s)$ for some $t\in B$, then either the
  fibration is isotrivial or $U=\{0\}$ in \eqref {Fujita}.
\end{teo}
\begin{proof}
  Set $Y=\phi\meno(s)$ and $H:= H^n(Y,\Zeta)$.  By \cite[p. 17-18]
  {moonen-ps} $\mt(H) =\go(H) $ for $n $ even and $\mt(H) =\gsp(H) $
  for $n $ odd.  Assume $d:=\rank U >0$.  Start with the case
  $d \neq \demi \rank H$.  Using \ref{sayspsp} and \ref{oo} we deduce
  that $K_d(H)$ is an irreducible representation of $\mt(H)$, hence an
  irreducible Hodge structure.  {If $f$ is non-isotrivial, \cite
    [Thm. 5.4] {peters-torelli}, implies that the IVHS map
    $T_{b}B \otimes H^{n,0}(X_b) \ra H^{n-1,1} (X_b)$ is non-zero for
    generic $b \in B$.  But then Theorem \ref {cor2} yields a
    contradiction.  This proves that $f$ is isotrivial.

  If $d = \demi \rank H$, then $U=V$ in \eqref{Fujita}, so again by
  \cite [Thm. 5.4] {peters-torelli}, the fibration is isotrivial.}
\end{proof}

Also in this case we can refine the previous statement making use of
the notion of Hodge locus.  Let $\mathcal{M}$ be the moduli space of
smooth complete intersections of multidegree $d_1, \lds, d_k$ in
$\PP^N$.

\begin{teo}
  Assume that $\sum_i d_i > N+1$ (i.e. the canonical bundle is ample).
  Let $\barf : \barX \ra \barb$ be a fibration as in
  \ref{mainsay}. Assume that $\barf$ is not isotrivial and that
  $d:=\rank U >0$. Then the image of $B$ in $\mathcal{M}$ is contained
  in a proper Hodge locus of $\mathcal{M}$.
\end{teo}
The proof is exactly as the one for Theorem \ref{teolocus}.

\section{Examples}
\label{sec:examples}

In this Section we compute explicitly the Hodge substructure given by
Corollary \ref{bau} in two examples belonging to the infinite family
constructed by Catanese and Dettweiler in \cite[\S\S 3-4] {cadett-3}
in order to get counterexamples to a question of Fujita. They are
families of cyclic covers of $\PP^1$.

\begin{ex}
  Consider the 1-dimensional family of $\Zeta/5$-covers of $\PP^1$
  given by the equation
  \begin{gather*}
    y^5 = x(x-1)(x+1) (x-t)^2,
  \end{gather*}
  where $t\in \PP^1 \setminus \{0,1,-1\}$.  The $\Zeta/5$ action is
  given by $(x,y) \mapsto (x, \zeta_5 y)$, where $\zeta_5$ is a fixed
  fifth primitive root of unity.

  The normalizations of these curves provide a family of smooth
  projective curves of genus $4$ parametrized by $\PP^1 -\{0,1,-1\}$.
  As shown in \cite[\S\S 3-4] {cadett-3} after a base change one gets
  a complete fibration $ \bar{X} \ra \bar{B}$ as in
  \ref{mainsay}. (This is one of the fibrations found by Catanese and
  Dettweiler in order to get counterexamples to a question of Fujita.)
  This family of curves also coincides with the family (11) in Table 1
  of \cite{moo}.  (See also Table 2 in \cite{fgp}.) It yields a
  Shimura curve in $\A_4$ that is generically contained in the Torelli
  locus.

  For every $t \in B$ we have a representation $\rho$ of
  $\Zeta/5 = \langle g \rangle$ on $H^0({X}_t, K_{{X}_t})$ and a
  decomposition in eigenspaces
  $H^0({X}_t, K_{{X}_t}) = \oplus_{i=1,...,4} V_i$, where
  $V_i = \{v \in H^0(X_t, K_{X_t}) \ | \ \rho(g)(v) = \zeta_5^i v\}$.
  Using eq. (3.2) in \cite{cadett-3} one easily computes
  \begin{gather*}
    \dim V_1 = 0, \ \dim V_2 = \dim V_3 = 1, \ \dim V_4 = 2.
  \end{gather*}
  In the Fujita decomposition \eqref{Fujita}, we have that $U_t = V_4$
  for every $t \in B$.  In fact, for every $t \in B$,
  $ H^1(X_t, \C) = \oplus_{j=1,...,4}{\mathbb H}_j$, where
  ${\mathbb H}_j$ is the $\zeta_5^j$-eigenspace for the action of
  $\Zeta/5$ on $H^1(X_t, \C)$. We have
  ${\mathbb H}_j = V_j \oplus \overline{V_{5-j}}$. Hence
  \begin{gather*}
    {\mathbb H}_1 = \overline{V_4}, \ {\mathbb H}_2 = V_2 \oplus
    \overline{V_3}, \ {\mathbb H}_3 = V_3 \oplus \overline{V_2}, \
    {\mathbb H}_4 = V_4.
  \end{gather*}
  Therefore
  $(f_* {\omega_{\bar{X}/B}})_t = H^0(X_t, K_{X_t}) = V_2 \oplus V_3
  \oplus V_4$
  and $V_4 = {\mathbb H}_4 = U_t$, while $A_t = V_2 \oplus V_3$, see \cite{cadett-3}.

  The decomposition $H^1(X_t, \C) = \oplus_{j=1,...,4} {\mathbb H}_j$
  is defined over $\QQ(\zeta_5)$, that is there are subspaces
  $\bombolo_j \subset H^1(X_t, \QQ(\zeta_5))$ such that
  \begin{gather*}
    H^1(X_t, \QQ(\zeta_5)) = \oplus_{j=1,...,4} \bombolo_j,
  \end{gather*}
  and $\mathbb{H}_j = \bombolo_j \otimes \C$.  A generator of the
  Galois group $G:=Gal(\QQ(\zeta_5), \QQ) \cong (\Zeta/5)^*$ is given
  by $h: \QQ(\zeta_5) \ra \QQ(\zeta_5)$ defined by
  $h(\zeta_5) = \zeta_5^2$. Since $H^1(X_t,\QQ(\zeta_5))$ is defined
  over $\QQ$ there is a natural a representation
  $\sigma : G \ra \Gl_\QQ (H^1(X_t,\QQ(\zeta_5)))$. We have
  $\sigma(h)(\bombolo_j) = \bombolo_{k}$, with $k \equiv 2j$ $\mod 5$.
  So $\sigma(h)(\bombolo_1) = \bombolo_{2}$,
  $\sigma(h)(\bombolo_2) = \bombolo_{4}$,
  $\sigma(h)(\bombolo_4) = \bombolo_{3}$,
  $\sigma(h)(\bombolo_3) = \bombolo_{1}$. Therefore
  $ \oplus_{j=1}^4\est^2 \bombolo_j $ is a $G$-invariant subspace of
  $\Lambda^2 { H^1(X_t, \QQ(\zeta_5))}$, hence it is defined over
  $\QQ$.  Thus
  \begin{gather*}
    H_t:= \oplus_{i=1,...,4} \Lambda^2 {\mathbb H}_i = \Lambda^2
    \overline{V_4} \oplus (V_2 \otimes \overline{V_3}) \oplus (V_3
    \otimes \overline{V_2}) \oplus \Lambda^2 V_4
  \end{gather*}
  is a Hodge substructure as in Corollary \ref{bau} and clearly
  $ \Lambda^2 U_t = \Lambda^2 \mathbb{H}_4 \subset H_t$.

  It is also easy to check that $H_t$
  is a proper substructure of $E_2 (H^1(X_t, \QQ))$.  In fact
  using the notation of \ref{sayspsp}, observe that $\om
  \in \est^2 H^1(X_t, \QQ)^*$ is the cup product and $\phi_2(s) = \om
  \lrcorner s = \om (s)$. Hence if $s = \alfa\wedge \beta
  $ with $\alfa \in \mathbb{H}_i$ and $\beta \in \mathbb{H}_j$, then
  \begin{gather*}
    \phi_2(s) = \int_C \alfa \wedge \beta = \int_C \zeta_5^*(\alfa
    \wedge \beta ) = \zeta_5 ^{i+j} \int_C \alfa \wedge \beta .
  \end{gather*}
  Thus if $i+j
  \not \equiv 0 \mod 5$, $ \phi_2(s) = 0$. So $H_t \subset
  \ker\phi_2$.  Next recall that $E_2(H^1(X_t,
  \QQ))$ is simply the orthogonal space to $\QQ \om$ inside $\est^2
  H^1(X_t, \QQ)$, so $\dim E_2 = 27$.  Thus clearly $\{0\}\subsetneqq
  H_t \subsetneqq E_2 (H^1(X_t, \QQ))$.

\end{ex}

\begin{ex}
  Consider the 1-dimensional family of $\Zeta/7$-covers of $\PP^1$
  given by the equation
$$y^7 = x(x-1)(x+1) (x-t)^4.$$ 
(This is one of the examples of \cite{cadett-3} and also family (17)
in Table 2 of \cite{moo}.)  The general fiber has genus 6.  Using the
same notation and by the same analysis as in the previous example one
gets
\begin{gather*}
  \dim(V_1) = \dim(V_2) = 0, \ \dim(V_3)= \dim(V_4) = 1,  \ \dim(V_5) = \dim(V_6) = 2,\\
  H^1(X_t, \C) = \oplus_{j=1,...,6} {\mathbb H}_j,\qquad
  U_t = V_5 \oplus V_6 = {\mathbb H}_5 \oplus {\mathbb H}_6,\\
  {\mathbb H}_1 = \overline{V_6}, \ {\mathbb H}_2 = \overline{V_5}, \
  {\mathbb H}_3 = V_3 \oplus \overline{V_4}, \ {\mathbb H}_4= V_4
  \oplus \overline{V_3}, \ {\mathbb H}_5 = V_5, \ {\mathbb H}_6 = V_6.
\end{gather*}
A generator of the Galois group
$Gal(\QQ(\zeta_7), \QQ) \cong (\Zeta/7)^*$ is
$h: \QQ(\zeta_7) \ra \QQ(\zeta_7)$, $h(\zeta_7) = \zeta_7^3$.  Hence
$\sigma(h)(\bombolo_j) = \bombolo_{k}$, with $k \equiv 3j$ $\mod
7$.
Therefore $\sigma(h)(\bombolo_1) = \bombolo_{3}$,
$\sigma(h)(\bombolo_3) = \bombolo_{2}$,
$\sigma(h)(\bombolo_2) = \bombolo_{6}$,
$\sigma(h)(\bombolo_6) = \bombolo_{4}$,
$\sigma(h)(\bombolo_4) = \bombolo_{5}$,
$\sigma(h)(\bombolo_5) = \bombolo_{1}$.  So the subspace
\begin{gather*}
  H_t:=   (\Lambda^2 {\mathbb H}_5 \otimes \Lambda^2 {\mathbb H}_6) \oplus  (\Lambda^2 {\mathbb H}_1 \otimes \Lambda^2 {\mathbb H}_4) \oplus  (\Lambda^2 {\mathbb H}_3 \otimes \Lambda^2 {\mathbb H}_5)\\
  \oplus (\Lambda^2 {\mathbb H}_2 \otimes \Lambda^2 {\mathbb H}_1)
  \oplus (\Lambda^2 {\mathbb H}_6 \otimes \Lambda^2 {\mathbb H}_3)
  \oplus (\Lambda^2 {\mathbb H}_4 \otimes \Lambda^2 {\mathbb H}_2)
\end{gather*}
is a Hodge substructure and
$\Lambda^4U_t = \Lambda^2 {\mathbb H}_5 \otimes \Lambda^2 {\mathbb
  H}_6 \subset H_t$.
Since $\om \lrcorner s =0$ for any
$s \in \mathbb{H}_i \wedge \mathbb{H}_j $ if
$i+j \not \equiv 0 \mod 7$, one easily checks that
$H_t \subset E_4(H^1(X_t,\QQ))$. By a dimension count $H_t$ is proper
substructure.
 
\end{ex}

\begin{remark}
  The families in the two examples above yield Shimura curves
  contained in the Torelli locus. This is not the case for all the
  other examples constructed by Catanese and Dettweiler, thanks to
  \cite{moo}.  Nevertheless in all the examples of Catanese and
  Dettweiler the bundle $U$ is non-trivial and by computations similar
  to the previous one, one can describe the Hodge substructure given
  by Corollary \ref{bau} .
\end{remark}

\section{Hodge loci in the moduli space of curves}
\label{sec:hodge-moduli}

We start by recalling some facts concerning totally geodesic
subvarieties and Hodge loci in $\Ag$.
\begin{say}
  \label{siegelsay}
  Let $\om$ be the standard symplectic form on $\R^{2g}$.  The
  \emph{Siegel space} $\Hg$ is the set of complex structures on
  $\R^{2g}$ that are compatible with $\om$, i.e. such that
  $J^*\om=\om$ and $\om (\cd, J \cd ) >0$.  The group $G:=\Sp(2g, \R)$
  acts on $\Hg$ by conjugation. This action is transitive. If
  $J \in \Hg$, the stabilizer $G_J$ is the group of unitary
  transformations of $(\R^{2g}, J, \om(\cd, J\cd))$.  Fix $J \in \Hg$
  and a unitary basis $\{e_1, \lds, e_g\}$ of $V^{1,0}(J)$.  Set
  $e_{g+j } := \bar{e}_j$. Then $\{e_1, \lds, e_{2g}\}$ is a basis of
  $\C^{2g}$. In this basis
  \begin{gather}
    \label{rapresent}
    \lieg := \Lie G = \Bigl \{
    \begin{pmatrix}
      X &  Y \\
      \bar{Y} & \bar{X}
    \end{pmatrix}
    : X \in \liu(g), Y\in \sime \Bigr\},
  \end{gather}
  where $\sime$ denotes the space of complex symmetric matrices of
  order $g$ (see e.g. \cite[p. 78-79]{satake-libro}).  Denote by
  $K=G_J$ the stabilizer of $J$ for the $G$-action. Let $\liem$ be the
  $\Ad K$-invariant complement of $\liek:=\Lie K$ in $\lieg$.  We have
  the Cartan decomposition $\lieg = \liek \oplus \liem$.  It is easy
  to check that
  \begin{gather}
    \label{lium}
    \liek = \Bigl \{
    \begin{pmatrix}
      X &  0 \\
      0 & \bar{X}
    \end{pmatrix} 
    \Bigr\} \cong \liu(g), \qquad \liem= \Bigl \{
    \begin{pmatrix}
      0 &  Y \\
      \bar{Y} & 0
    \end{pmatrix} 
    \cong \sime.
  \end{gather}
  It follows that $\Hg$ is a Hermitian symmetric space of the
  noncompact type.  In terms of the identifications \eqref{lium} the
  isotropy representation of $\liek$ on $\liem$
  becomes 
  \begin{gather}
    \label{adk}
    \ad_\liek : \liek \ra \gl (\liem 
    ) , \quad
    \ad_\liek(X) (Y) = XY -Y\bar{X}.
  \end{gather}
\end{say}

\begin{say}
  \label{saysimsub}
  The fact that $\Hg$ is a symmetric space of the noncompact type has
  important consequences for its totally geodesic submanifolds. On the
  one hand these are necessarily symmetric spaces of their own.  On
  the other hand the closed totally geodesic submanifolds of $\Hg$
  passing through $J$ are in bijective correspondence with Lie triple
  systems $\liel$, that is with linear subspaces $\liel \subset \liem$
  such that $[[\liel, \liel],\liel] \subset \liel$, see
  e.g. \cite[p. 237]{KN-II}.
\end{say}

\begin{say}\label{agsay}

  The group $\Ga: = \Sp(2g, \Zeta)$ acts properly discontinuously and
  holomorphically on $\Hg$.  Hence $\Ag:=\Ga \backslash \Hg$ is a
  complex analytic global quotient orbifold. The symmetric metric of
  $\Hg$ descends to an orbifold K\"ahler metric on $\Ag$.  We will
  always consider this metric on $\Ag$. It is a locally symmetric
  (orbifold) metric.
\end{say}

\begin{say}
  \label{sayorbag}
  In Riemannian geometry a submanifold $N $ of a Riemannian manifold
  $(M,g)$ is called \emph{totally geodesic} if the second fundamental
  form of $N$ in $M$ vanishes identically.  We use the same
  terminology for suborbifolds of a Riemannian orbifold.  {More
    precisely, we will say that a subset $Z \subset \Ag$ is a
    \emph{totally geodesic subvariety}, if it is a closed algebraic
    subvariety of $\Ag$ and there is a totally geodesic submanifold
    $\tilde{Z}$ of $\Hg$ such that $\pi(\tilde{Z}) = Z$.}

  It has been proved by Mumford that special subvarieties, i.e. Hodge
  loci of $\Ag$, are totally geodesic,
  \label{mumford}
  see \cite{mumford-Shimura} and \cite{moonen-linearity-1}.

\end{say}

\begin{say}
  \label{saydefphik}
  For $1\leq k \leq g-1$ consider the map
  $ \phi_k : \A_k \times \A_{g-k} \lra \A_g$,
  \begin{gather}
    \label{zebra}
    \phi_k ([A_1, \chern_1(L_1)], [A_2, \chern_1(L_2) ]):= [ A_1
    \times A_2 , \chern_1( L_1 \boxtimes L_2) ].
  \end{gather}
  We will often drop the polarizations from the notation.  Set
  \begin{gather*}
    Z_k:=\phi_k(\A_k \times \A_{g-k}).
  \end{gather*}
\end{say}

\begin{prop}\label{oops}
  \begin{enumerate}[label=\emph{\alph*})]
  \item \label{oo1} $Z_k$ is a totally geodesic subvariety of $\A_g$.
  \item \label{oo2} $Z_k$ is maximal in the following sense: if $Z$ is
    a totally geodesic subvariety of $\Ag$ and $Z_k \subset Z$, then
    either $Z=Z_k$ or $Z=\Ag$.
  \item \label{oo3} Fix $[A_0]\in \A_k$ and set
    $ h : \A_{g-k} \lra \A_g $, $ h ([A]):= [A_0\times A]$.  If
    $Z\subset \A_g$ is a totally geodesic subvariety, then
    $h\meno (Z) $ is a totally geodesic subvariety of $\A_{g-k}$.
  \end{enumerate}
\end{prop}
\begin{proof}
  \ref{oo1}
  Fix $[J] \in Z_k $. Then
  $(\R^{2g}/\Zeta^{2g}, J) \cong A_1 \times A_2$. Set
  $V_i = T_0 A_i \subset \R^{2g}$.  Then $J(V_i) = V_i$. Moreover
  there are sublattices $\La_i \subset \Zeta^{2g}$, such that
  $V_i = \La_i\otimes \R$, $\La_1 \oplus \La_2 = \Zeta^{2g}$ and
  $\om\restr{\La_i} $ is a principal polarization i.e. a form of type
  $(1,\lds, 1)$.  Set $G':=\{ a \in \Sp(2g, \R): a (V_i)=V_i $ for
  $i=1,2\}^0$.  Fix a basis $\{e_1, \lds, e_k\}$ of $V_1^{1,0}$ and a
  basis $\{e_{k+1}, \lds, e_g\}$ of $V_2^{1,0}$.  Using the basis
  $\{e_1 , \lds, e_g\}$ as in \ref{siegelsay}, the matrix
  representation \eqref{rapresent} and the identifications
  \eqref{lium} we see that $\lieg' = \liek' \oplus \liem'$ , where
  \begin{equation}
    \label{liekprimo}     
    \begin{gathered}
      \liek'=\Bigl \{X =
      \begin{pmatrix}
        X_1 & 0 \\ 0 & X_2
      \end{pmatrix} , \quad X_1 \in \liu(k),\quad X_2 \in \liu(g-k)\Bigr\},\\
      \liem'= \Bigl \{ Y =
      \begin{pmatrix}
        Y_1 & 0 \\ 0 & Y_2
      \end{pmatrix}, \quad Y_1 \in \Sym_{k},\quad Y_2 \in \Sym_{g-k}
      \Bigr \} .
    \end{gathered}
  \end{equation}
  Thus $G'\cong \Sp(2k,\R)\times \Sp(2g -2k,\R)$ and
  $G'_J = \U(k)\times \U(g-k)$. It follows that
  $\tilde{Z}_k:= G' \cd J$ is a totally geodesic submanifold of $\Hg$
  isometric to $\sieg_k\times \sieg_{g-k}$.
{Clearly $\pi (\tilde{Z}_k) = Z_k$.}
  This proves \ref{oo1}.

\noindent
\ref{oo2} It is enough to check that $\tilde{Z}_k$ is a maximal
totally geodesic submanifold of $\Hg$.  Totally geodesic submanifolds
are in bijective correspondence with Lie triple systems in $ \liem$
see \ref{saysimsub}.  Thus it is enough to prove that the Lie triple
system $\liem' $ in \eqref{liekprimo} corresponding to $\tilde{Z}_k$
is a maximal Lie triple system in $\liem$.  Denote by $\uu$ the vector
space of complex $(g-k) \times k$ matrices.  The orthogonal complement
of $\liem'$ in $\liem$ with respect to the Killing form (which is a
multiple of the trace) is the space
\begin{gather}
  \liem'' =\Bigl \{ Y =
  \label{liemdue}
  \begin{pmatrix}
    0 & \tras{\xi} \\ \xi & 0
  \end{pmatrix}, \xi \in \uu\Bigr \}.
\end{gather}
Identify $\liem''$ with $\uu$ by the correspondence
$Y \leftrightarrow \xi$.  Then, using \eqref{adk}, for $X \in \liek'$
as in \eqref{liekprimo} and $Y \in \liem'$ as in \eqref{liemdue} we
get
\begin{gather*}
  \ad_\liek(X)(Y) = X_2 \xi - \xi \bar{X}_1 =X_2\xi + \xi\tras{X_1} .
\end{gather*}
Thus the representation of $\liek'$ on $\liem'$ reduces to the
representation
\begin{gather*}
  \label{repliu}
  \liu(k) \oplus \liu(g-k) \lra \gl (\uu),\quad (X_1, X_2) \cd \xi =
  X_2 \xi + \xi \cd\tras{X_1} .
\end{gather*}
This is an irreducible representation, since it is the outer tensor
product of the standard representations of $\liu(k)$ and $\liu(g-k)$,
see e.g. \cite[p. 197]{goodman-wallach}.  Thus $\liem''$ is an
irreducible $\liek'$-module.  Assume now that $\liem'''$ is a Lie
triple system such that $\liem ' \subset \liem ''' \subset \liem$.
Set $\lieq:= \liem''\cap \liem'''$. Then
\begin{gather*}
  [\liek ' , \lieq ] = [[\liem', \liem'], \lieq ] \subset [[ \liem''',
  \liem '''], \liem''']\subset \liem'''.
\end{gather*}
Moreover
$ [\liek ' , \lieq ] \subset [\liek ' , \liem'' ] \subset \liem''$.
Thus $[\liek ' , \lieq ] \subset \lieq$, i.e. $\lieq$ is a
$\liek'$-submodule of $\liem''$. Since $\liem''$ is an irreducible
$\liek'$-module, there are two possibilities: either $\lieq = \{0\}$
and $\liem'''=\liem'$, or $\lieq=\liem''$ and $\liem ''' = \liem$.
These possibilities correspond to $\tilde{Z}=\tilde{Z_k}$ and
$\tilde{Z} = \sieg_g$ respectively.  Thus $\tilde{Z}_k$ and $Z_k$ are
indeed maximal.

\noindent
\ref{oo3} Set $W:=h ( \A_{g-k})$.  Assume that
$A_0 \cong (V_0/\La_0, J_0, \om_0)$.  Choose complementary sublattices
$\La_i \subset \Zeta^{2g}$ such that $\om \restr{\La_i}$ is principal
and there is a symplectic isomorphism
$\eta : (\La_0, \om_0) \ra (\La_1, \om\restr{\La_1})$. Set
$V_i : = \La_i \otimes \R$.  Set
$\tilde{W} : = \{ J' \in \Hg: J(V_i) = V_i $ for $i=1,2$ and
$J\restr{V_1} = \eta J_0 \eta\meno\}$. Then  $\pi(\tilde{W}) =W $, so  $W$ is a
totally geodesic subvariety of $\Ag$.  Moreover $h$ admits a lifting
$\tilde{h}: \sieg_{g-k} \ra \Hg$ with image exactly $\tilde{W}$ and
this lifting is clearly an isometric immersion. Next let $Z\subset \Ag$ be
a totally geodesic subvariety. If $Z\cap W = \vacuo$ we have nothing
to prove. Otherwise we can choose a totally geodesic submanifold 
$\tilde{Z} \subset \Hg$ such that $\pi(\tilde{Z} )= Z$ and
$\tilde{W} \cap \tilde{Z} \neq \vacuo$. This intersection is totally
geodesic. Since $\tilde{h}$ is isometric,
$\tilde{h}\meno(\tilde{W} \cap \tilde{Z})$ is totally geodesic in
$\sieg_{g-k}$ and it is a connected component of
$\pi\meno (h\meno(Z))$. This proves \ref{oo3}.

\end{proof}

\begin{say}
  Let $\M_g$ and $\mb_g$ denote the moduli space of curves of genus
  $g$ and its Deligne-Mumford compactification.  Let $\dodo$ denote
  the set of points in $\mb_g$ that represent stable curves of the
  form $ E \cup_p C$, where $E$ is a smooth elliptic curve, $C$ is a
  smooth curve of genus $g-1$, $p\in C$ and the notation $E \cup_p C$
  means that $0\in E$ and $p\in C$ are identified.  Then
  $\Delta_1 = \overline{\dodo}$.  The map
  \begin{equation}
    \label{def-psi}
    \psi:      \M_{1,1} \times \M_{g-1,1} \lra \dodo,
    \qquad \psi([E],[C,p]) := [E\cup_pC],
  \end{equation}
  is an isomorphism.  Let
  $ \pi_1 :\M_{1,1}\times \M_{g-1,1} \ra \M_{1,1}$ and
  $ \pi_2 :\M_{1,1}\times \M_{g-1,1} \ra \M_{g-1,1}$ be the
  projections. Set
  \begin{gather}
    \label{def-bi}
    \bi_1:= \pi_1 \circ \psi\meno: \dodo \ra \M_{1,1}, \qquad \bi_2:=
    \pi_2 \circ \psi\meno: \dodo \ra \M_{g-1,1}.
  \end{gather}
\end{say}

The following result uses the same argument as in \cite[\S
4]{marcucci-naranjo-pirola}.
\begin{teo}
  \label{pietro}
  Assume $g\geq 4$.  If $Y\subset \M_g$ is an irreducible divisor and
  $\bary$ is its closure in $\mb_g$, then
  $\bi_1(\bary \cap \dodo) = \M_{1,1}$.
\end{teo}
\begin{proof}
  Both $\Mg$ and $\Mb_g$ have quotient singularities and are therefore
  $\QQ$-factorial.   So we can find a line bundle $L \ra \mb$, a section
  $s\in H^0(\mb, L)$ and an integer $m>0$ such that $m\bary $ is the
  zero divisor of $s$.

  A basis for $\Pic(\mb_g)\otimes \QQ$ is given by
  $\{\la, \delta_0, \lds, \delta_{[g/2]}\}$, where $\la$ denotes the
  determinant of the Hodge bundle and $\delta_i$ are the boundary
  divisors. (These are not line bundles on $\mg$, but on the moduli
  stack $\overline{\mathcal{M}}_g$.  We are interested in properties
  that do not change when a divisor/line bundle is multiplied by a
  positive integer. So this is no harm.)

  In $\Pic(\mb_g)\otimes \QQ$ we have
  $ L \equiv a \la + \sum_id_i \delta_i$ for some $a, d_i \in \QQ$.
  It is well-known that $a\neq 0$.  (Given $x \in Y_{\mathrm{reg}}$
  there is a complete curve $C \subset \M_g$ such that $x \in C$ and
  $T_xC \pitchfork T_xY$, see \cite[Rmk. 4.1,
  p. 431]{marcucci-naranjo-pirola}.  Hence
  $\deg(L\restr{C}) = a \cd \deg (\la\restr{C}) >0$.)

  Since $g-1 \geq 3$ we can find a complete curve
  $B \subset \M_{g-1}$.  Next we fix an arbitrary elliptic curve
  $E$. For $b\in B$ denote by $\Ga_b$ the smooth curve corresponding
  to the moduli point $b$. For $p\in \Ga_b$ consider the nodal curve
  $\Ga_b \cup_p E$ obtained by gluing the points $p\in \Ga_b$ and
  $ 0 \in E$.  Varying $p$ and $b$ we get a complete surface
  $S \subset \dodo$.  Let $C \subset S$ be the curve obtained by
  fixing a particular value $b_0 \in B$.  Observe that
  $\Delta_i \cap S \neq \vacuo$ only for $i=1$.  Thus
  $\delta_i\restr{S} = 0$ for $i\neq 1$ and
  \begin{gather*}
    L\restr {S} \equiv a \la \restr{S} + d_1 \delta_1\restr{S}.
  \end{gather*}
  Moreveor $\la\restr{C} = 0$, since the Hodge structure does not vary
  on $C$.  Thus
  \begin{gather*}
    L\restr {C} \equiv d_1 \delta_1\restr{C}.
  \end{gather*}
  We claim that $\bary $ meets $ S$.  By contradiction assume that
  $\bary \cap S = \vacuo$.  Then $L\restr{S} \equiv 0$, so also
  $L\restr{C} \equiv 0$.  On the other hand there is a line bundle
  with non-zero degree on $C$, since $\mb_g$ is projective. Hence
  $\deg (\delta _1 \restr{C}) \neq 0 $.  It follows that $d_1=0$.  But
  then $L\restr{S} = a \la\restr{S} \equiv 0$.  Since $a \neq 0$, this
  implies that $\la\restr{S} \equiv 0$.  {But this is false:
    denote by $p_3 : \Delta_1^0 \ra \M_{g-1}$ the composition of
    $p_2 $ with the obvious projection $\M_{g-1,1} \ra \M_{g-1}$. Then
    $\la\restr{S} \cong (\bi_1^* \la_{\M_{1,1}} \otimes
    p_3^*\la_{\M_{g-1}})\restr{S}$ and $p_3\restr{S}: S \ra B$ has
    connected fibers, so $ (p_3)_* 
    (\la\restr{S}) = \la_{M_{g-1}}\restr{B}$.  Since
    $\la_{M_{g-1}} \cd B >0$, $\la\restr{S} \not \equiv 0$.} We have
  proved that $\bary \cap S \neq \vacuo$.  If
  $x \in \bary \cap S \subset \bary \cap \dodo$, then
  $\bi_1(x) = [E]$. Thus $[E] \in \bi_1 (\bary\cap \dodo)$. Since $E$
  is arbitrary the theorem is proved.  \mihi { $\Pic(\mathcal{M}_g)$
    is free abelian.  $\Pic(\mb)$ is a finite index subgroup: p. 381.
    Satake compactification: p. 435-437.  $\QQ$-factorial:
    Kollar-Mori, p.4.  Quotient singularities are $\QQ$-factorial:
    ibidem, p.158.  }
\end{proof}

The following definition is standard in Riemannian geometry. In some
sense it is the Riemannian analogue of the notion of non-degenerate
projective variety.
\begin{defin}
  An analytic subset $X \subset \Ag$ is \emph{full} if there is no
  proper totally geodesic subvariety $Z \subsetneqq \Ag$ that contains
  $X$.
\end{defin}

\begin{say}
  \label{bo}
  Let $\tilde{Z} \subset \Hg$ be a totally geodesic submanifold. Then
  the real codimension of $\tilde{Z} $ in $\Hg$ is at least $g$.
  Indeed by a theorem of Berndt and Olmos \cite{berndt-olmos} the real
  codimension of a totally geodesic submanifold of a Riemannian
  symmetric space is at least the rank of the symmetric space. Since
  the rank of $\Hg$ is $g$ the result follows immediately.
\end{say}

\begin{teo}
  \label{notg}
  Let $j : \Mg \ra \Ag$ be the period map.  If $g\geq 3$ and
  $Y \subset \M_g$ is an irreducible divisor, then $j(Y)$ is full in
  $\A_g$.
\end{teo}
\begin{proof}
  Assume first that $g=3$. Then $\dim\M_3 =\dim\A_3= 6$. If
  $Y \subset \M_3$ is a hypersurface, also $j(Y)$ is a hypersurface.
  If $j(Y) $ is contained in a proper totally geodesic subvariety
  $Z$, this is also a hypersurface. Then $\tilde{Z} \subset \sieg_3 $
  would be a totally geodesic submanifold of real codimension 2. This
  is impossible by the theorem of Berndt and Olmos quoted above.

  We proceed by induction on $g$.  Assume that the result holds for
  $g-1$ and that $g\geq 4$.  Let $Z \subset \Ag$ be a totally geodesic
  subvariety and assume by contradiction that there is an algebraic
  hypersurface $Y\subset \M_g$ such that $j(Y) \subset Z$.  We want to
  prove that $Z=\Ag$.  The period map $j$ extends to
  $ \barj : \Mb_g - \Delta_0 \ra \Ag$ and
  $ \barj (\bary -\Delta_0) \subset Z$.

  We claim that for any $[E]\in \A_1$ there is an irreducible divisor
  $Y_E \subset \M_{g-1} $ such that
  \begin{gather*}
    \{[E]\} \times j(Y_E) \subset j(\bary - \Delta_0).
  \end{gather*}

  Let $\psi: \A_1\times \M_{g-1,1} \ra \Delta_1^0 \subset \Mb_g$ be the
  map $\psi([E], [C,p]) = [E\cup_p C]$ as in \eqref{def-psi}.  By
  Theorem \ref{pietro} the intersection $\bary \cap \Delta_1^0$ is
  non-empty and has dimension $3g -5$.  Since $\psi$ is an
  isomorphism, the set $W:= \psi\meno (\bary \cap \Delta_1^0) $ is a
  closed algebraic subset of $ \M_{1,1}\times \M_{g-1,1}$ of dimension
  $3g -5$.  Let
  \begin{gather*}
    \pi_1 : \M_{1,1}\times \M_{g-1,1} \ra \M_{1,1}, \qquad
    \pi_2  : \M_{1,1}\times \M_{g-1,1}  \ra \M_{g-1,1},  \\
    \pi_3 : \M_{g-1,1} \lra \M_{g-1} , \qquad q:= \pi_1\restr{W} : W
    \lra \M_{1,1},
  \end{gather*}
  be the obvious projections.  Fix $[E] \in \M_{1,1}$.  By Theorem
  \ref{pietro} there is $[C,p] \in \M_{g-1,1}$ such that
  $[E \cup _p C ] \in \bary$.  Then
  $q\meno ([E]) = \{ [E]\} \times W'$ for a closed subset
  $W' \subset \M_{g-1,1}$ with $\dim W' \geq 3g -6$.  Set
  $W'':= \pi_3 (W') \subset \M_{g-1}$.  Since $\dim W''\geq 3g-7$,
  ${W''}$ contains the generic point of some irreducible divisor $Y_E$
  of $\M_{g-1} $.  By construction for any $[C] \in W''$, there is
  $p\in C$ such that $[E\cup_p C ] \in W$.  Moreover we have
  $\barj ([E\cup_p C]) = [E \times J(C)]$ as polarized abelian
  varieties.  Thus
  $[E \times J(C)] = j ([E\cup_p C]) \in j(W) \subset j(\bary -
  \Delta_0)$.
  This holds in particular for the generic point of the divisor
  $Y_E$. Hence $\{[E]\}\times j (Y_E) \subset j (\bary - \Delta_0)$.
  The claim is proved.

  Now fix $[E] \in \A_1$ and set
  \begin{gather*}
    h : \A_{g-1} \lra \A_g , \qquad h ([A]):= [E\times A].
  \end{gather*}
  Since $ \barj (\bary -\Delta_0) \subset Z$, we have
  $h \circ j (Y_E) = \{[E]\}\times j(Y_E) \subset Z$.  Thus
  \begin{gather*}
    j (Y_E) \subset h\meno (Z).
  \end{gather*}
  By Proposition \ref{oops} \ref{oo3} $h\meno(Z)$ is a totally
  geodesic subvariety of $\A_{g-1}$ and it contains $j(Y_E)$.  By the
  inductive hypothesis $j(Y_E)$ is full in $\A_{g-1}$.  So we conclude
  that $h\meno (Z) = \A_{g-1}$.  Using the notation of \eqref{zebra}
  this means that $ \phi_1 (\{[E]\} \times \A_{g-1}) \subset Z$.  But
  $[E] \in \A_1$ is arbitrary, so we have in fact
  $ Z_1= \phi_1 (\A_1 \times \A_{g-1} ) \subset Z$.  By Proposition
  \ref{oops}\ref{oo2} $Z_1$ is a maximal totally geodesic subvariety
  of $\Ag$. Therefore either $Z=Z_1$ or $Z=\Ag$. The first possibility
  is absurd since $Z$ contains by hypothesis the jacobians of smooth
  curves. Hence we have proved that $Z=\Ag$, i.e. that $j(Y)$ is full.
\end{proof}

\begin{remark}
  This result of course implies that the Jacobian locus $j(\Mg)$
  itself is full for $g\geq 3$. This can be proved directly (and
  easily) using the same argument.
\end{remark}
\begin{remark}
  The second fundamental form of $\Mg$ in $\Ag$ has been studied in
  \cite{cpt} using the Hodge-Gaussian maps (see also \cite {pi}). It follows from the
  analysis in that paper that the second fundamental form is non-zero
  along Schiffer variations.  (See also \cite{cf1,cf2} for related
  results.)  Using this it has been proven in
  \cite[Thm. 4.4]{colombo-frediani-ghigi} that any totally geodesic
  submanifold of $\Ag$ that is generically contained in $j(\Mg)$ has
  dimension at most $ \frac{5}{2}(g-1) $.  In particular if
  $Y \subset \M_4$ a hypersurface, then $j(Y)$ is not totally geodesic
  in $\Ag$.  This yields a different proof of the theorem for $g=4$.
  Indeed assume by contradiction that $Z \subset \Ag$ is a proper
  totally geodesic subvariety and that $j(Y) \subset Z$.  By the
  theorem of Berndt and Olmos mentioned in \ref{bo},
  $\dim_\R Z \leq 16$.  Since $\dim_\R Y = 16$, $j(Y) $ would be open
  in $Z$, so $j(Y)$ would be totally geodesic.  But then we should
  have $8=\dim j( Y ) \leq \frac{5}{2} (g-1) = \frac{15}{2}$, a
  contradiction.
\end{remark}

\begin{cor}
  \label{hlmg}
  Let $g \geq 3$.  Let $Z \subset \A_g$ be a special subvariety,
  i.e. a Hodge locus for the canonical variation of Hodge structure on
  $\A_g$. Then $ j\meno (Z)$ has codimension at least two in $\M_g$.
\end{cor}
\begin{proof}
  This follows immediately from Theorem \ref{notg} using the theorem
  of Mumford mentioned in \ref{mumford}.
\end{proof}

\begin{remark}
  It is well-known that for a very general $[A]$ in $ \Ag$ we have
  $\mt(A) = \gsp(H^1(A))$.  By Corollary \ref{hlmg} if $[C]\in \Mg$ is
  very general, then $j([C])$ belongs to no proper Hodge locus of
  $\Ag$, hence $\mt(C) = \gsp(H^1(C))$. This argument yields another
  proof of Proposition \ref{vgmt}.
\end{remark}

\begin{remark}
  It might be interesting to notice that the previous results give
  some information about the monodromy along a divisor.  Indeed let
  $Y\subset \Mg$ be an irreducible divisor and let $y\in Y$ be a very
  general point.  Let $C_y$ be a curve with moduli point $y$.  By the
  Corollary \ref{hlmg} $\mt(C_y) = \gsp(2g)$.  Denote by $\Gamma_y$
  the monodromy group at $y$ of the variation of Hodge structure over
  $Y$. Let $\overline{\Gamma}_y$ be the Zariski closure of $\Gamma_y$
  inside $\GL(g,\QQ)$ and let $H_y:=\overline{\Gamma}_y^0$ be the
  connected component of the identity.  By a result of Andr\'e
  \cite[Thm. 1]{andre}, $H_y$ is a normal subgroup of the commutator
  subgroup of $\mt(y)$.  (In Andr\'e's paper this result is formulated
  for variations of mixed Hodge structure. The statement in the pure
  case is simpler, see \cite[Thm. 16]{peters-steenbrink-monodromy}.)
  Since $\Sp(g,\QQ)$ is simple, this means that $H_y = \Sp(g,\QQ) $.
  Thus a finite index subgroup of the monodromy along $Y$ is Zariski
  dense in $\Sp(g,\QQ)$.
\end{remark}

\def\cprime{$'$}

\end{document}